\documentclass[oneside,english]{amsart}
\usepackage[T1]{fontenc}
\usepackage[utf8]{inputenc}
\usepackage{array}
\usepackage{multirow}
\usepackage{amstext}
\usepackage{amsthm}
\usepackage{amsmath}
\usepackage{amssymb}
\usepackage{hyperref}

\usepackage[table]{xcolor}

\usepackage[
backend=biber,
style=alphabetic,
sorting=ynt
]{biblatex}
\makeatletter

\providecommand{\tabularnewline}{\\}

\numberwithin{equation}{section}
\numberwithin{figure}{section}
\theoremstyle{plain}
\newtheorem{thm}{\protect\theoremname}
  \theoremstyle{definition}
  
  \theoremstyle{remark}
  
  \theoremstyle{plain}
  \newtheorem{conjecture}[thm]{\protect\conjecturename}
  \newtheorem{lem}[thm]{\protect\lemmaname}
  \theoremstyle{definition}
  
  \theoremstyle{plain}
  \newtheorem{cor}[thm]{\protect\corollaryname}
  \providecommand{\corollaryname}{Corollary}

\makeatother

\usepackage{babel}
  \providecommand{\conjecturename}{Conjecture}
  \providecommand{\definitionname}{Definition}
  \providecommand{\examplename}{Example}
  \providecommand{\remarkname}{Remark}
\providecommand{\lemmaname}{Lemma}
\providecommand{\theoremname}{Theorem}

\begin{document}

\title{On Primorial Numbers}

\author{Jonatan Gomez \\ jgomezpe@unal.edu.co \\ Universidad Nacional de Colombia}
\maketitle
\begin{abstract}
    Prime numbers have attracted the attention of mathematicians and enthusiasts for millenniums due to their simple definition and remarkable properties. In this paper, we study primorial numbers (the product of the first prime numbers) to define primorial sets, primorial intervals, primorial tables, and primorial totative numbers. We establish relationships between prime numbers and primorial totative numbers and between admissible $k$-tuples of prime numbers and admissible $k$-tuples of primorial totative. Finally, we study the Goldbach conjecture and derive four Goldbach conjectures using primorial intervals, twin, cousin, and sexy prime numbers.
\end{abstract}

\section{Introduction}
\subsection{Prime Numbers}
A prime number is a natural number having exactly two factors, $1$ and itself \cite{narkiewicz2000}. The natural number $1$ is not considered a prime number since it has only one factor ($1$). From now on, $p_{n}$ will denote the $n$\textbf{\emph{th}} prime number, here $n\geq 1$. The first eleven prime numbers are $2,3,5,7,11,13,17,19,23,29,31$. Although different classes of prime numbers have been defined \cite{wikipediaprimelist}, in this paper, we especially concentrate on some prime number classes derived from the notions of admissible $k$-tuple and $k$-constellation \cite{10.1090/S0025-5718-99-01117-5,wikipediaprimektuple}. 
\subsection{Admissible tuples}
\label{sec:admissible}
A $k$-tuple $a=(a_1,a_2,\ldots,a_{k})$ of natural numbers is admissible if and only if $a_1=0$, $a_{i-1}<a_i$ for all $i=2,3,\ldots,k$, and there is a prime number $p>3$ such that $(p,p+a_2,\ldots,p+a_k)$ is a $k$-tuple of prime numbers, i.e. $p+a_i$ is a prime number for all $i=1,2,\ldots,k$. We say that a prime number $p>3$ satisfies an admissible $k$-tuple $a$ (we denote this situation as $[p|a]$) if and only if $(p,p+a_2,\ldots,p+a_k)$ is a $k$-tuple of prime numbers. The diameter of an admissible $k$-tuple $a$ is $dia(a)=a_k$ while its gap is the maximum difference between two of its consecutive elements, i.e. $gap(a)=max\left\{a_i-a_{i-1} \, \middle| \, i=2,3,\ldots,k\right\}$. For example, 
\begin{enumerate}
    \item $a=(0,2)$ is an admissible $2$-tuple. If $p=5$, then $5+0=5$ and $5+2=7$, so $(5,7)$ is a $2$-tuple of prime numbers. Clearly, $dia(a)=a_2=2$ and $gap(a)=max\{2-0\}=2$.
    \item $a=(0,2,6)$ is an admissible $3$-tuple. If $p=5$ then $5+0=5$, $5+2=7$, and $5+6=11$, so $(5,7,11)$ is a $3$-tuple of prime numbers. Clearly, $dia(a)=a_2=6$ and $gap(a)=max\{2-0,6-2\}=4$.
    \item $(0,1)$ is not admissible: since $p>3$ is a prime number then $p=2x+1$ (odd number, i.e., not a multiple of $2$). Then $p+1=2x+1+1=2x+2=2(x+1)$ is an even number, therefore it is not a prime number (contradiction).
    \item $(0,2,4)$ is not admissible $3$-tuple: since $p>3$ then $p=3x+y$ with $y=1,2$ ($p$ is prime so it is not a multiple of $3$). If $y=1$ then $p+2=3x+1+2=3x+3)3(x+1)$ so $p+2$ is a multiple of $3$ then it is not a prime number (contradiction). If $y=2$ then $p+4=3x+2+4=3x+6=3(x+2)$. So $p+4$ is a multiple of $3$ then it is not a prime number (contradiction).
\end{enumerate} 
\subsection{Constellations} 
\label{sec:constellation}
An admissible $k$-tuple $a=(a_1,a_2,\ldots,a_{k})$ is $k$-constellation if and only if for any admissible $k$-tuple $b=(b_1,b_2,\ldots,b_{k})$ we have that $dia(a)\leq dia(b)$. For example, $(0,2)$, $(0,4)$, and $(0,6)$ are admissible $2$-tuples but only $(0,2)$ is a $2$-constellation ($2\leq 4\leq 6$).

\subsection{Prime Number Classes}
\label{sec:primeclasses}
The following is a short list of prime number classes defined in terms of admissible $k$-tuples and $k$-constellations:

\begin{enumerate}
    \item \textbf{Twin primes.} Two prime numbers $p$ and $q$ are twin primes if and only if $q=p+2$. For example, primes $5$ and $7$ are twin primes since they define the prime $2$-tuple $(5,5+2)$ which satisfies the $2$-prime constellation $(0,2)$.
    \item \textbf{Cousin primes.} Two prime numbers $p$ and $q$ are called cousin primes if and only if $q=p+4$. For example, primes $7$ and $11$ are cousin primes since they define the prime $2$-tuple $(7,7+4)$ which satisfies the admissible $2$-tuple $(0,4)$.
    \item \textbf{Sexy primes.} Two prime numbers $p$ and $q$ are called sexy primes if and only if $q=p+6$. For example, primes $5$ and $11$ are sexy primes since they define the prime $2$-tuple $(5,5+6)$ which satisfies the admissible $2$-tuple $(0,6)$.
    \item \textbf{Sexy primes triplet.} Three prime numbers $p$, $q$, and $r$ are called a sexy prime triplet if and only if $q=p+6$ and $r=p+12$. For example, primes $5$, $11$, and $17$ are a sexy prime triplet since they define the prime $3$-tuple $(5,5+6,5+12)$ which satisfies the admissible $3$-tuple $(0,6,12)$.
    \item \textbf{Sexy primes quadruplet.} Four prime numbers $p,q,r,$ and $s$ are called a sexy prime quadruple if and only if $q=p+6$, $r=p+12$, and $s=p+18$. For example, primes $5$, $11$, $17$, and $23$ are a sexy prime quadruplet since they define the prime $4$-tuple $(5,5+6,5+12,5+18)$ which satisfies the admissible $4$-tuple $(0,6,12,18)$.
\end{enumerate}

\subsection{\label{Primorial}Primorial}

Many interesting operations can be defined over prime numbers, for example, we can define the \textbf{primorial} of the $n$\emph{th} prime number as the product of the first $n\in\mathbb{N}^{+}$ prime numbers \cite{Dubner1987}, i.e., $\#(n)=p_n\#=\prod_{i=1}^{n}p_k$. This definition is similar to the definition of the factorial function, so it is possible to define the primorial as a recursive function $\#(0)=1$ and $\#(n)=p_n\#(n-1)$ for $n\geq 1$. The first six primorials are $1, 2,6,30,210,2310$.

\subsection{Coprime or Relative-prime Numbers}

Two natural numbers $a, b \in \mathbb{N}$ are coprime or relative-prime numbers if and only if their greatest common divisor is $1$ ($gcd(a,b)=1$). According to this definition, \textit{i)} $1$ is relative-prime with any other positive natural number, \textit{ii)} $0$ is not relative-prime with any natural number, and \textit{iii)} any prime number is relative-prime with any other prime number. Notice that we can define prime numbers in terms of coprime numbers: A natural number $p>1$ is a prime number if and only if $p$ is relative-prime to $q$ for all natural numbers $1\leq q<p$.

\subsection{Natural numbers less than $n$ ($\mathbb{Z}_n$)}
Let $n\in\mathbb{N}$, the set of natural numbers less than $n$, is $\mathbb{Z}_n=\{0,1,\ldots,n-1\}$. For example, $\mathbb{Z}_2=\{0,1\}$ and $\mathbb{Z}_{\#(2)}=\mathbb{Z}_{2*3}=\mathbb{Z}_{6}=\{0,1,3,4,5\}$. Amount the properties hold by $\mathbb{Z}_n$, we are especially interested in the structure of the subset of natural numbers that are relative-prime to $\#(n)$ (sometimes called as totative numbers of $\#(n)$). For instance, consider $\mathbb{Z}_{\#(2)}=\mathbb{Z}_6$, the subset of $\mathbb{Z}_6$ defined by the natural numbers that are relative-prime to $6$ is $\{1,5\}$. It is clear that any prime number $q$ such $p_n<q<\#(n)$ will be a relative-prime number to $\#(n)$. 

\subsection{Euler's totient function ($\varphi(n)$)}
Euler's totient function \cite{Euler1763} counts the natural numbers, in $\mathbb{Z}_n$, which are relative-prime to $n$. Take for example $\mathbb{Z}_{\#(3)}=\mathbb{Z}_{30}$, the subset of natural numbers relative-prime to $30$ is $\{1,7,11,13,17,19,23,29\}$, therefore $\varphi(30)=8$. Euler's totient function is a multiplicative function, i.e., $\varphi(ab)=\varphi(a)\varphi(b)$ for two coprime numbers $a$ and $b$. Moreover, for a prime number $p$, we have that $\mathbb{Z}_p - \{0\}$ is the set of totative numbers of $p$, therefore $\varphi(p)= p-1$. Using these two properties of Euler's totient function, we can easily compute it on primorial numbers: $\varphi(\#(n))=\prod_{i=1}^{n}(p_k-1)=\prod_{i=1}^{n}\varphi(p_k)$.

\section{Primorial Sets, Primorial Intervals, and Primorial Totatives}

\subsection{Primorial sets}
We define a $n$-primorial set by 'replacing' numbers $0$ and $1$ in $\mathbb{Z}_{\#(n)}$, with numbers $\#(n)$ and $\#(n)+1$, respectively. We do this since the number $1$ is always coprime to another positive natural number. Let $n\in\mathbb{N}$, the $n$-primorial set is defined as  $\mathbb{Z}^{\#}_n=\{2,3,\ldots,\#(n),\#(n)+1\}$. The following is the list of the first five $n$-primorial sets:

\begin{enumerate}
    \item $\mathbb{Z}^{\#}_1=\{2,3\}$
    \item $\mathbb{Z}^{\#}_2=\{2,3,4,5,6,7\}$
    \item $\mathbb{Z}^{\#}_3=\{2,3,\ldots,30,31\}$
    \item $\mathbb{Z}^{\#}_4=\{2,3,\ldots,210,211\}$
    \item $\mathbb{Z}^{\#}_5=\{2,3,\ldots,2310,2311\}$    
\end{enumerate}

\subsection{Primorial intervals}
Let $n\in\mathbb{N}$, the $n$-primorial interval is defined as $\mathcal{I}_n=\left\{m \, \middle| \, \#(n-1)+1 \leq x \leq \#(n)+1\right \}$. The following is the list of the first five $n$-primorial intervals:

\begin{enumerate}
    \item $\mathcal{I}_1=\{2,3\}$
    \item $\mathcal{I}_2=\{3,4,5,6,7\}$
    \item $\mathcal{I}_3=\{7,8,\ldots,30,31\}$
    \item $\mathcal{I}_4=\{31,32,\ldots,210,211\}$
    \item $\mathcal{I}_5=\{211,212,\ldots,2310,2311\}$    
\end{enumerate}

Notice that primorial intervals are not disjoint sets: the last element of $\mathcal{I}_n$ is the first element of $\mathcal{I}_{n+1}$.

\subsection{\label{sec:tot}Primorial $n$-totatives}

We can extend the notion of totatives of $\#(n)$ to set $\mathbb{Z}^{\#}_n$ by considering the subset of natural numbers in $\mathbb{Z}^{\#}_n$ being relative-prime to $\#(n)$. The following is the list of the first three $n$-totative sets: 

\begin{enumerate}
    \item The $1$-totative set is $tot(1)=\{3\}$
    \item The $2$-totative set is $tot(2)=\{5,7\}$
    \item The $3$-totative set is $tot(3)=\{7,11,13,17,19,23,29,31\}$
\end{enumerate}

Notice that any prime number $p \in \mathbb{Z}^{\#}_n$ must be either $p_i$ for some $i=1,2,\ldots,n$ or a $n$-totative number, i.e., we can express the Prime Number Theorem in terms of primorials and $n$-totative numbers. We will explore this relationship in a future paper. In this section, we just extend definitions in sections \ref{sec:admissible}-\ref{sec:primeclasses} to $n$-totative numbers.

\subsection{$n$-admissible} 
A $k$-tuple $a=(a_1,a_2,\ldots,a_{k})$ of natural numbers is $n$-admissible if and only if $a_1=0$, $a_{i-1}<a_i$ for all $i=2,3,\ldots,k$, and for all $m\geq n$ there is a $m$-totative number $t_m$ such that $(t_m,t_m+a_2,\ldots,t_m+a_k)$ is a $m$-totative $k$-tuple. Notice that a $n$-admissible $k$-tuple is also a $m$-admissible $k$-tuple for all $m>n$. A $n$-totative number $t$ \textbf{satisfies} an $n$-admissible $k$-tuple $a$ if and only if $a[t]=(t,t+a_2,\ldots,t+a_k)$ is a $k$-tuple of $n$-totative numbers. The \textbf{gap} of a $k$ tuple $a$ is the maximum difference between two consecutive elements of the $k$-tuple, i.e., $gap(a)=max\left\{a_i-a_{i-1} \, \middle| \, i=2,3,\ldots,k\right\}$.

\subsection{$n$-strong}
A $n$-admissible $k$-tuple $a=(a_1,a_2,\ldots,a_{k})$ is $n$-strong if and only if for all $m\geq n$ we have that $mod(a_i,p_m)\neq mod(a_j,p_m)$ with $i=1,\ldots,k-1$ and $j=i+1,\ldots,k$. Notice that being $n$-strong implies that $k\leq p_n$.

\subsection{$n$-constellation}
A $n$-admissible $k$-tuple $a=(a_1,a_2,\ldots,a_{k})$ is $n$-totative $k$-constellation if and only if for $a_k\leq b_k$ for any $n$-admissible $k$-tuple $b=(b_1,b_2,\ldots,b_{k})$.

\subsection{Classes of primorial $n$-totative numbers}

The following is a short list of $n$-totative numbers classes:

\begin{enumerate}
    \item \textbf{Isolated.} Let $a=(0,a_2,\ldots,a_{k})$ a $n$-admissible $k$-tuple, and $t$ a $n$-totative number such that $a[t]$ is a $n$-totative $k$-tuple. $a[t]$ is isolated if and only if both $a[t-a_2]$ and $a=[t_2]$ are not $n$-totative $k$-tuples. 
    \item \textbf{$n$-Twin.} Two $n$-totative numbers $p$ and $q$ define a $n$-twin couple if and only if $q=p+2$. Notice that a $n$-twin couple satisfies the $n$-strong $n$-totative $2$-constellation $(0,2)$. Moreover, it is an isolated $2$-tuple when $n>2$ since the $3$-tuple $(0,2,4)$ is not a $n$-admissible $3$-tuple: $t$, $t+2$, and $t+4$ all must be coprime to $3$, therefore $t=3s+1$ or $t=3s+2$ for some natural number $s$. If $t=3s+1$ then $t+2=3s+1+2=3(s+1)$, so $t+2$ is not coprime to $3$. Now if $t=3s+2$ then $t+4=3s+2+4=3(s+2)$, so $t+4$ is not coprime to $3$. 
    \item \textbf{$n$-Cousin.} Two $n$-totative numbers $p$ and $q$ are a $n$-cousin couple if and only if $q=p+4$. Notice that a $n$-cousin couple satisfies the $n$-strong $n$-admissible $2$-tuple $(0,4)$. Moreover, it is an isolated $2$-tuple when $n>3$ ($(0,4,8)$ is not a $n$-admissible $3$-tuple, we left the proof to the reader).
    \item \textbf{$n$-Sexy.} Two $n$-totative numbers $p$ and $q$ are a $n$-sexy couple if and only if $q=p+6$. Notice that a $n$-sexy couple $(p,p+6)$ satisfies the $n$-admissible $2$-tuple $(0,6)$. Moreover, it is isolated if and only if $p-6$ and $p+12$ are not $n$-totative.
    \item \textbf{$n$-Sexy triplet.} Three $n$-totative numbers $p$, $q$, and $r$ are a $n$-sexy triplet iff $q=p+6$ and $r=p+12$. Notice that a $n$-sexy triplet $(p,p+6,p+12)$  satisfies the strong $n$-admissible $3$-tuple $(0,6,12)$. Moreover, it is isolated if and only if $p-6$ and $p+18$ are not $n$-totative.
    \item \textbf{$n$-Sexy primes quadruplet.} Four $n$-totative numbers $p$, $q$, $r$, and $s$ are a $n$-sexy quadruplet iff $q=p+6$, $r=p+12$, and $s=p+18$. Notice that any $n$-sexy quadruplet  satisfies the $n$-strong $n$-admissible $4$-tuple $(0,6,12,18)$. Moreover, it is an isolated $4$-tuple ($(0,6,12,18,24)$ is a non $n$-admissible $5$-tuple, we left the proof to the reader).
\end{enumerate}

\section{Primorial Tables}
Primorial sets can be arranged as tables, see Tables \ref{tab:1-primorial}, \ref{tab:2-primorial}, and \ref{tab:3-primorial}. Notice that \textit{i)} Multiples of prime numbers $p<p_n$ are located on columns, of the $n$-primorial table, columns that are completely marked as red, i.e., a column with elements that are not relative-primes to $\#(n-1)$; \textit{ii)} Prime $p_n$ will mark just one element of each column (multiple of $p_n$); and \textit{iii)} There are columns maintaining almost exclusively the $n$-totative numbers, except made for the one in each column that is a multiple of $p_n$. We call these columns $n$-totative columns. Table \ref{tab:4-totient-columns} shows just the $4$-totative columns. 

\begin{table}[htbp]   
    \centering
    \begin{tabular}{|c|}
        \hline
        \cellcolor{blue!25}2 \\
        \hline
        3 \\ 
        \hline
    \end{tabular}
    \caption{\label{tab:1-primorial}$1$-primorial table ($\#(1)=2$). Multiples of $p_1=2$ are shown in blue. $1$-totatives are shown in white.}
\end{table}    
    
\begin{table}[htbp]
    \centering
    \begin{tabular}{|c|c|}
        \hline
        \cellcolor{red!25}2 & \cellcolor{blue!25}3  \\
        \hline
        \cellcolor{red!25}4 & 5 \\
        \hline
        \cellcolor{yellow}6 & 7 \\
        \hline    
    \end{tabular}
    \caption{\label{tab:2-primorial}$2$-primorial table ($\#(2)=2*3=6$). Multiples of $p_2=3$ are shown in blue, multiples of $p_1=2$ are shown in red, and multiples of $p_1=2$ and $p_2=3$ are shown in yellow. $2$-totatives are shown in white.}
\end{table}    

\begin{table}[htbp]
    \centering
    \begin{tabular}{|c|c|c|c|c|c|}
        \hline
        \cellcolor{red!25}2 & \cellcolor{red!25}3 & \cellcolor{red!25}4 & \cellcolor{blue!25}5 & \cellcolor{red!25}6 & 7 \\
        \hline
        \cellcolor{red!25}8 & \cellcolor{red!25}9 & \cellcolor{yellow}10 & 11 & \cellcolor{red!25}12 & 13 \\
        \hline
        \cellcolor{red!25}14 & \cellcolor{yellow}15 & \cellcolor{red!25}16 & 17 & \cellcolor{red!25}18 & 19 \\
        \hline
        \cellcolor{yellow}20 & \cellcolor{red!25}21 & \cellcolor{red!25}22 & 23 & \cellcolor{red!25}24 & \cellcolor{blue!25}25 \\
        \hline
        \cellcolor{red!25}26 & \cellcolor{red!25}27 & \cellcolor{red!25}28 & 29 & \cellcolor{yellow}30 & 31 \\    
        \hline
    \end{tabular}
    \caption{\label{tab:3-primorial}$3$-primorial table ($\#(3)=2*3*5=30$). Multiples of $p_3=5$ are shown in blue, multiples of $p_1=2$ or $p_2=3$ are shown in red, and multiples of $p_1=2$ or $p_2=3$, and $p_3=5$ are shown in yellow. $3$-totatives are shown in white.}
\end{table}    

\begin{table}[htbp]
    \centering
    \begin{tabular}{|c|c|c|c|c|c|c|c|}
        \hline
        \cellcolor{blue!25}7 & 11 & 13 & 17 & 19 & 23 & 29 & 31 \\
        \hline
        37 & 41 & 43 & 47 & \cellcolor{blue!25}49 & 53 & 59 & 61 \\
        \hline
        67 & 71 & 73 & \cellcolor{blue!25}77 & 79 & 83 & 89 & \cellcolor{blue!25}91 \\
        \hline
        97 & 101 & 103 & 107 & 109 & 113 & \cellcolor{blue!25}119 & 121 \\
        \hline
        127 & 131 & \cellcolor{blue!25}133 & 137 & 139 & 143 & 149 & 151 \\
        \hline
        157 & \cellcolor{blue!25}161 & 163 & 167 & 169 & 173 & 179 & 181 \\
        \hline
        187 & 191 & 193 & 197 & 199 & \cellcolor{blue!25}203 & 209 & 211 \\
        \hline
    \end{tabular}
    \caption{\label{tab:4-totient-columns}$4$-totative columns ($\#(4)=2*3*5*7=210$). Multiples of $p_4=7$ are shown in blue.}
\end{table}    

We define rows and columns in a $n$-primorial table as follow: the $k$th row of a $n$-primorial table is $row(n,k)=\{m\mid m=k*\#(n-1)+j$ for $2\leq j\leq\#(n-1)+1\}$ with $0\leq k < p_n$. The $n$-totative column defined by the $(n-1)$-totative number $t$ is $col(n,t)=\{m\mid m=t+k*\#(n-1)$ for $0\leq k < p_n\}$. For example, the $4$-totative column $col(4,13)=\{m\mid m=13+k*30$ for $0\leq k < 7\}$, i.e., $col(4,13)=\{13,43,73,103,133,163,193\}$, is defined by the $3$-totative number $13$. 

\begin{lem}
\label{lem:mark-one}
   There is one and just one element of $col(n,t)$ that is a multiple of $p_n$.
\end{lem} 
\begin{proof}
We need to prove that each element of $col(n,t)$ produces a different residue module $p_n$. Suppose that there are two elements of $col(n,t)$ producing the same residue $c$ module $p_n$, i.e., two elements $t_1 = t+k_1*\#(n-1)=a*p_n+c$ and $t_2=t+k_2*\#(n-1)=b*p_n+c$ such that $0\leq k_1 < k_2 < p_n$ and $0 \leq c < p_n$. If we subtract the first number from the second we have that $(k_2-k_1)*\#(n-1)=(b-a)*p_n$ then $\#(n-1)$ must be a multiple of $p_n$ since $0<k_2-k_1<p_n$, but it is impossible by definition of $\#(n-1)$. Now, we have $p_n$ different elements in $col(n,t)$ so we need $p_n$ different residues module $p_n$, therefore, there is one residue that must be $0$, i.e., there is only one element of $col(n,t)$ that is a multiple of $p_n$.
\end{proof}
\begin{cor}
\label{cor:tot-col}
If $n>1$ and $t$ is $n$-totative number then there is a $(n-1)$-totative number $t'$ such that $t$ belongs to the $n$-totative column defined by $t'$, i.e., $t=t'+k*\#(n-1)$ for some $0\leq k < p_n$.
\end{cor}
\begin{proof}
For any $t$ in the $n$-primorial set, there is a $t'$ in the $(n-1)$ primorial set and $0\leq k < p_n$ such that $t=t'+k*\#(n-1)$ (definition of $n$-table rows). Suppose that $t'$ is not a $(n-1)$-totative number, then $gcd(t',\#(n-1)) > 1$. Therefore, there is a prime number $p<p_n$ such that $t'=p*s$ for some $s\geq 1$. Using definition of $t$ we have that $t=p*s+p*\frac{\#(n-1)}{p}=p*\left(s+\frac{\#(n-1)}{p}\right)$. Since $p$ also divides $\#(n)$ then  $gcd(t,\#(n))\geq p$ but this is impossible since $t$ is $n$-totative, i.e., $gcd(\#(n), t)=1$. Therefore $t'$ is a $\#(n-1)$-totative number.
\end{proof}

\section{Counting Primorial $n$-Totative Numbers}
By using primorial tables, we can count $n$-totative numbers and their different classes. Numeric results are obtained with the C++ program (\textit{onprimorials.cpp}) freely available at the prime numbers github repository of professor Jonatan Gómez \cite{GomezPrimesGit}.

\begin{thm}
\label{thm:tot-phi}
The number of $n$-totatives can be defined recursively as $tot(1)=1$ and $tot(n)=(p_n-1)*tot(n-1)$ for $n>1$.
\end{thm}
\begin{proof}
Clearly $tot(1)=1$, see Table \ref{tab:1-primorial}. Since there is only one multiple of $p_n$ in each $n$-totient column such an element will be marked (it is non $n$-totative), then there are $(p_n-1)*tot(n-1)$ elements that are $n$-totative.
\end{proof}

Result in Theorem \ref{thm:tot-phi} is expected since $\#(n)$ and $(\#(n)+1)$ are coprime numbers, therefore the number of $n$-totatives is equal to the totatives of $\#(n)$, i.e, the number of $n$-totatives is $\varphi(\#(n))=\prod_{i=1}^{n}\varphi(p_k)=\prod_{i=1}^{n}(p_k-1)=tot(n)$. However, we will consider the idea of 'marking' just one element of the $n$-totative columns for counting $n$-totative $k$-tuples, and $k$-constellations. Notice that any prime number in the $n$-primorial set must be a $n$-totative number or $p_i$ for $i=1,2,\ldots,n$. Table \ref{tab:tot-pi} shows this relation between the count of $n$-totative numbers ($tot(n)$) and the total number of prime numbers in the $n$-primorial set, i.e., up to $\#(n)+1$ ($\pi(\#(n)+1)$).

\begin{table}[htbp]
\centering
\begin{tabular}{|r|r|r|r|r|}
\hline
$n$ & $\#(n)$ & $tot(n)$ & $\pi(\#(n)+1)$ & $\frac{tot(n)}{\pi(\#(n)+1)}$\\
\hline
3 & 30 & 8 & 11 & 0.727273\\
\hline
4 & 210 & 48 & 47 & 1.02128\\
\hline
5 & 2310 & 480 & 344 & 1.39535\\
\hline
6 & 30030 & 5760 & 3248 & 1.7734\\
\hline
7 & 510510 & 92160 & 42331 & 2.17713\\
\hline
8 & 9699690 & 1658880 & 646029 & 2.56781\\
\hline
9 & 223092870 & 36495360 & 12283531 & 2.97108\\
\hline
10 & 6469693230 & 1021870080 & 300369796 & 3.40204\\
\hline
\end{tabular}
\caption{\label{tab:tot-pi} Relation between the count of $n$-totative numbers ($tot(n)$) and the total number of prime numbers in the $n$-primorial set, i.e., up to $\#(n)+1$ ($\pi(\#(n)+1)$).}
\end{table}

\subsection{Admissible Tuples of Totatives}
For determining the number of different $n$-totative number classes, we need some technical Theorems about $n$-admissible $k$-truples.

\begin{thm}  
\label{thm:tuplestable}
   Let $a=(0,a_2,\ldots,a_k)$ be a $(n-1)$-admissible $k$-tuple such that $gap(a)<p_n-1$ and $t$ a $n$-totative number. If $t$ satisfies $a$ then there is a $(n-1)$-totative number $t'$ such that $t=t'+\#(n-1)*j$ for some $j=0,1,\ldots,p_n-1$ and $t'$ satisfies $a$.
\end{thm} 
\begin{proof}The first $n$-totative column is generated by prime $p_n$ (any natural number greater than $1$ and lower than $p_n$ must be a multiple of some $p_i$ with $i=1,2,\ldots,n-1$) and the last $n$-totative column is generated by $\#(n-1)+1$. Therefore, the gap between the element in the last $n$-totative column of row $r$ and the first element in the $n$-totative column of row $r+1$ is $(\#(n-1)*r+p_n)-(\#(n-1)*r+1)=\#(n-1)-1=p_n-1$. Since $gap(a)<p_n-1$ it is impossible that the element in the last $n$-totative column of row $r$ and the element in the first column of row $r+1$ satisfy the $k$-tuple $a$. Therefore, it is impossible that a $n$-totative number $t$ defines a $n$-admissible $k$-tuple (with $gap(a)<p_n-1$) distributed in two or more different rows, i.e, $t+a_i\leq \#(n-1)*(j+1)+1$ for all $i=1,2,\ldots,k$. According to Corollary \ref{cor:tot-col}, $t=t'+\#(n-1)j$ for some $j=0,1,\ldots,p_n-1$, then $t+a_i=t'+a_i+\#(n-1)*j \leq \#(n-1)*(j+1)+1$, so $t'+a_i\leq \#(n-1)+1$, i.e., $t'+a_i$ is $(n-1)$-totative for all $i=1,2,\ldots,k$.
\end{proof}

\begin{thm}  
\label{thm:tuplescount}
   Let $a=(0,a_2,\ldots,a_k)$ be a strong $(n-1)$-admissible $k$-tuple such that $gap(a)<p_n-1$. For any isolated $(n-1)$-totative $k$-tuple satisfying $a$ there are $p_n-k$ isolated $n$-totative $k$-tuples satisfying $a$.
\end{thm} 
\begin{proof}If $t$ is a $n-1$-totative number such that $t$ satisfies $a$ and $a[t]$ is an isolated $(n-1)$-totative $k$-tuple then there are $p_n$ candidate isolated $n$-totative $k$-tuples (one per each row of the $n$-primorial table on the same columns of the isolated $(n-1)$-totative $k$-tuple. Since $p_n$ marks only one element in a single $n$-totative column defined by $a$ (Lemma \ref{lem:mark-one}) and $p_n$ cannot mark two different columns $t+a_i$ and $t+a_j$ $i\neq j$ ($a$ is strong), $p_n$ marks $k$ different candidate $n$-admissible $k$-tuples generated by $a[t]$. Therefore, $count_a(m)=count_a(m-1)*(p_m-k)$.
\end{proof}

\subsection{Twin Totative Numbers and Twin Primes}
We can now count twin $n$-totatives and compare them against twin primes.

\begin{cor}
\label{cor:n-twin}
   The number of $n$-twin couples ($twin(n)$) can be defined recursively as $twin(1)=0$, $twin(2)=1$ and $twin(n)=(p_n-2)*twin(n-1)$ for $n>2$.
\end{cor} 
\begin{proof}
Clearly $twin(1)=0$ and $twin(2)=1$, see Tables \ref{tab:1-primorial} and \ref{tab:2-primorial} respectively. Notice that $gap\left((0,2)\right)=2$ and $p_3=5$, so $gap\left((0,2)\right)<p_3-1$. Since $(0,2)$ is a strong $3$-admissible $2$-tuple and every $n$-twin couple is isolated for $n>2$, then by Theorem \ref{thm:tuplescount} we have that $twin(n)=(p_n-2)*twin(n-1)$ for all $n>2$. 
\end{proof}

Table \ref{tab:twin} shows the relationship between the count of $n$-totative twin couples ($twin(n)$) and the total number of twin prime couples in the $n$-primorial set, i.e., up to $\#(n)+1$ ($twin_{*}(\#(n)+1)$).

\begin{table}[htbp]
\centering
\begin{tabular}{|r|r|r|r|r|}
\hline
$n$ & $\#(n)$ & $twin(n)$ & $twin_{*}(\#(n)+1)$ & $\frac{twin(n)}{twin_{*}(\#(n)+1)}$\\
\hline
3 & 30 & 3 & 5 & 0.6\\
\hline
4 & 210 & 15 & 15 & 1\\
\hline
5 & 2310 & 135 & 70 & 1.92857\\
\hline
6 & 30030 & 1485 & 468 & 3.17308\\
\hline
7 & 510510 & 22275 & 4636 & 4.80479\\
\hline
8 & 9699690 & 378675 & 57453 & 6.59104\\
\hline
9 & 223092870 & 7952175 & 896062 & 8.87458\\
\hline
10 & 6469693230 & 214708725 & 18463713 & 11.6287\\
\hline
\end{tabular}
\caption{\label{tab:twin} Relation between the count of $n$-totative twin couples ($twin(n)$) and the total number of twin prime couples the $n$-primorial set, i.e., up to $\#(n)+1$ ($twin_{*}(\#(n)+1)$).}
\end{table}

\subsection{Cousin Totative Numbers and Cousin Primes}
We can now count cousin $n$-totatives and compare them against cousin primes. 

 \begin{cor}
\label{cor:n-cousin}
   The number of $n$-cousin couples ($cousin(n)$) can be defined recursively as $cousin(1)=cousin(2)=0$, $cousin(3)=3$, and $cousin(n)=(p_n-2)*cousin(n-1)$ for $n>3$.
\end{cor} 
\begin{proof}
Clearly $cousin(1)=0$, $cousin(2)=0$  and $cousin(3)=3$, see Tables \ref{tab:1-primorial} to \ref{tab:3-primorial} respectively. Notice that $gap\left((0,4)\right)=4$ and $p_4=7$, so $gap\left((0,4)\right)<p_4-1$. Since $(0,4)$ is a strong $4$-admissible $2$-tuple and every $n$-cousin couple is isolated for $n>3$, then by Theorem \ref{thm:tuplescount} we have that $cousin(n)=(p_n-2)*cousin(n-1)$ for all $n>3$.
\end{proof}

\begin{cor}
$twin(n)=cousin(n)$ for all $n\geq3$.
\end{cor}

Table \ref{tab:cousin} shows the relationship between the count of $n$-totative cousin couples ($cousin(n)$) and the total number of cousin prime couples in the $n$-primorial set, i.e., up to $\#(n)+1$ ($cousin_{*}(\#(n)+1)$).

\begin{table}[htbp]
\centering
\begin{tabular}{|r|r|r|r|r|}
\hline
$n$ & $\#(n)$ & $cousin(n)$ & $cousin_{*}(\#(n)+1)$ & $\frac{cousin(n)}{cousin_{*}(\#(n)+1)}$\\
\hline
3 & 30 & 3 & 4 & 0.75\\
\hline
4 & 210 & 15 & 14 & 1.07143\\
\hline
5 & 2310 & 135 & 71 & 1.90141\\
\hline
6 & 30030 & 1485 & 468 & 3.17308\\
\hline
7 & 510510 & 22275 & 4630 & 4.81102\\
\hline
8 & 9699690 & 378675 & 57065 & 6.63585\\
\hline
9 & 223092870 & 7952175 & 896737 & 8.8679\\
\hline
10 & 6469693230 & 214708725 & 18460319 & 11.6308\\
\hline
\end{tabular}
\caption{\label{tab:cousin} Relation between the count of $n$-totative cousin couples ($cousin(n)$) and the total number of cousin prime couples the $n$-primorial set, i.e., up to $\#(n)+1$ ($cousin_{*}(\#(n)+1)$).}\end{table}

\subsection{Sexy Totative Numbers and Sexy Primes}
We can now count sexy $n$-totatives and compare them against sexy primes. 

\begin{cor}
\label{cor:n-sexy}
   The number of $n$-sexy quadruplets ($quad(n)$) can be defined recursively as $quad(1)=quad(2)=0$, $quad(3)=1$, $quad(4)=6$, and $quad(n)=(p_n-4)*quad(n-1)$ for $n>4$.
\end{cor} 
\begin{proof} Clearly $quad(1)=quad(2)=0$, $quad(3)=1$  and $quad(4)=6$, see Tables \ref{tab:1-primorial} to \ref{tab:4-totient-columns} respectively.  Notice that $gap\left((0,6,12,18)\right)=6$ and $p_5=11$, so $gap\left((0,6,12,18)\right)<p_5-1$. Since $(0,6,12,18)$ is a strong $5$-admissible $4$-tuple and every $n$-quadruple is isolated for $n>4$, then by Theorem \ref{thm:tuplescount} we have that $quad(n)=(p_n-4)*quad(n-1)$ for all $n>4$.
\end{proof}

\begin{lem}
\label{lem:n-sexy}
   The number of isolated $n$-sexy triplets ($itriple(n)$), $n$-sexy triplets ($triple(n)$), isolated $n$-sexy couples ($isexy(n)$), and $n$-sexy couples ($sexy(n)$) can be defined as follow:
   \begin{enumerate}
       \item $itriple(1)=itriple(2)=0$, $itriple(3)=1$, $itriple(4)=4$, and $itriple(n)=(p_n-3)*itriple(n-1)+2*quad(n-1)$ for $n>4$.
       \item $triple(n)=itriple(n)+2*quad(n)$ for all $n>0$.
       \item $isexy(1)=isexy(2)=isexy(3)=0$, $isexy(4)=4$, and $isexy(n)=(p_n-2)*isexy(n-1)+2*itriple(n-1)+2*quad(n-1)$ for $n>4$.
       \item $sexy(n)=isexy(n)+2*itriple(n)+3*quad(n)$.
   \end{enumerate}
\end{lem} 
\begin{proof} We left proofs to the reader (just consider which columns are marked by prime number $p_n$ in each case).
\end{proof}
\begin{cor}
\label{cor:n-sexy}
    The number of $n$-sexy couples ($sexy(n)$) can be defined recursively as $sexy(1)=sexy(2)=0$, $sexy(3)=5$, $sexy(4)=30$, and $sexy(n)=(p_n-2)*sexy(n-1)$ for all $n>4$.
\end{cor}
\begin{proof}
    Clearly, $sexy(1)=sexy(2)=0$, $sexy(3)=5$, and $sexy(4)=30$ by computing $sexy(n)$ with definition 4) of Lemma \ref{lem:n-sexy}. Now, by item 4) of Lemma \ref{lem:n-sexy} we have $(p_n-2)*sexy(n-1)=(p_n-2)*\left(isexy(n-1)+2*itriple(n-1)+3*quad(n-1)\right)$. By distributing the product on the right side, then expressing $2(p_n-2)$ as $2(p_n-3)+2$ and $3(p_n-2)$ as $3(p_n-4)+2+4$ and organizing terms we have $(p_n-2)*sexy(n-1)=(p_n-2)*isexy(n-1)+2*itriple(n-1)+2*quad(n-1)+2*(p_n-3)*itriple(n-1)+4*quad(n-1)+3*(p_n-4)*quad(n-1)$. Now, by grouping terms $1-3$, $4-5$, and using Lemma \ref{lem:n-sexy} we have $(p_n-2)*sexy(n-1)=isexy(n)+2*itriple(n)+3*quad(n)=sexy(n)$.
\end{proof}
\begin{cor}
    $sexy(n)=2*twin(n)$ for all $n>=4$.
\end{cor}

Table \ref{tab:sexy} shows the relationship between the count of $n$-totative sexy couples ($sexy(n)$) and the total number of sexy prime couples in the $n$-primorial set, i.e., up to $\#(n)+1$ ($sexy_{*}(\#(n)+1)$). Notice that the behavior of the ratio between the number of twin primes and the number of $n$-twins is the same as the behavior of the ratio between the number of cousin primes and the number of $n$-cousins, and the behavior of the ratio between the number of sexy primes and the number of $n$-sexy numbers. 

\begin{table}[htbp]
\centering
\begin{tabular}{|r|r|r|r|r|}
\hline
$n$ & $\#(n)$ & $sexy(n)$ & $sexy_{*}(\#(n)+1)$ & $\frac{sexy(n)}{sexy_{*}(\#(n)+1)}$\\
\hline
3 & 30 & 5 & 6 & 0.833333\\
\hline
4 & 210 & 30 & 26 & 1.15385\\
\hline
5 & 2310 & 270 & 140 & 1.92857\\
\hline
6 & 30030 & 2970 & 951 & 3.12303\\
\hline
7 & 510510 & 44550 & 9331 & 4.77441\\
\hline
8 & 9699690 & 757350 & 114189 & 6.63243\\
\hline
9 & 223092870 & 15904350 & 1792173 & 8.87434\\
\hline
10 & 6469693230 & 429417450 & 36921295 & 11.6306\\
\hline
\end{tabular}
\caption{\label{tab:sexy}Relation between the count of $n$-totative sexy couples ($sexy(n)$) and the total number of sexy prime couples in the $n$-primorial set, i.e., up to $\#(n)+1$ ($sexy_{*}(\#(n)+1)$).}
\end{table}

\section{Goldbach Conjecture}
There are a lot of different conjectures about prime numbers \cite{wikipediaprimeconjectures}. In particular, Goldbach's conjecture stays that any even natural number greater than six ($6$) can be expressed as the sum of two prime numbers \cite{Rassias,https://doi.org/10.1112/plms/s2-44.4.307}. In this section, we derive a version of Goldbach's conjecture but for primorial intervals and combine Goldbach's conjecture with twin, cousin, and sexy prime numbers. From now on, $\mathcal{P}$ denotes the set of prime numbers, $\mathcal{E}$ denotes the set of even natural numbers, $\mathcal{P}_{n}$ denotes the set of primes in the $n$-primorial interval, i.e., $\mathcal{P}_{n}=\mathcal{P}\bigcap \mathcal{I}_{n}$, and $\mathcal{E}_{n}$ denotes the set of even number in the $n$-primorial interval, i.e., $\mathcal{E}_{n}=\mathcal{E}\bigcap \mathcal{I}_{n}$.

\begin{conjecture}
\textbf{\label{conj:(Binary-Goldbach)}(Goldbach)} For any positive even number natural $m\geq4$ there exist at least two prime numbers $p,q\in P$ such that $m=p+q$. 
\end{conjecture}

\begin{conjecture}
\label{conj:intervals-Goldbach}\textbf{(Goldbach-Intervals)}
For any positive natural number $n$ and any $m\in \mathcal{E}_{n+1}$ there are at least two prime numbers $p\in\mathcal{P}_{n}$ and $q\in\mathcal{P}_{n}\bigcup\mathcal{P}_{n+1}$ such that $m=p+q$.
\end{conjecture}

The pair of prime numbers satisfying the Goldbach-Intervals conjecture for the first two primorial intervals are shown in Table \ref{tab:First-two-Interval}. Goldbach-Intervals conjecture for the first $10$-primorial intervals can be validated with the C++ program (\textit{onprimorials.cpp}) freely available at the prime numbers github repository of professor Jonatan Gómez \cite{GomezPrimesGit}.

\begin{table}
\begin{centering}
\begin{tabular}{|c|c|c|c|c|c|}
\hline 
$n$ & $\mathcal{P}_{n}$ & $\mathcal{P}_{n}\bigcup\mathcal{P}_{n+1}$ & $m\in E_{n+1}$ & $p$ & $q$\tabularnewline
\hline 
\hline 
\multirow{2}{*}{$1$} & \multirow{2}{*}{$\left\{ 2,3\right\} $} & \multirow{2}{*}{$\left\{ 2,3,5,7\right\} $} & $4$ & $2$ & $2$\tabularnewline
\cline{4-6} 
 &  &  & $6$ & $3$ & $3$\tabularnewline
\hline 
\multirow{12}{*}{$2$} & \multirow{12}{*}{$\left\{ 3,5,7\right\} $} & \multirow{12}{*}{$\left\{ 3,5,7,11,13,17,19,23,29,31\right\} $} & $8$ & $3$ & $5$\tabularnewline
\cline{4-6} 
 &  &  & $10$ & $3$ & $7$\tabularnewline
\cline{4-6} 
 &  &  & $12$ & $5$ & $7$\tabularnewline
\cline{4-6} 
 &  &  & $14$ & $3$ & $11$\tabularnewline
\cline{4-6} 
 &  &  & $16$ & $3$ & $13$\tabularnewline
\cline{4-6} 
 &  &  & $18$ & $5$ & $13$\tabularnewline
\cline{4-6} 
 &  &  & $20$ & $3$ & $17$\tabularnewline
\cline{4-6} 
 &  &  & $22$ & $3$ & $19$\tabularnewline
\cline{4-6} 
 &  &  & $24$ & $5$ & $19$\tabularnewline
\cline{4-6} 
 &  &  & $26$ & $3$ & $23$\tabularnewline
\cline{4-6} 
 &  &  & $28$ & $5$ & $23$\tabularnewline
\cline{4-6} 
 &  &  & $30$ & $7$ & $23$\tabularnewline
\hline 
\end{tabular}
\par\end{centering}
\caption{\label{tab:First-two-Interval}Validation of the Goldbach-Intervals conjecture for the first two primorial intervals.}

\end{table}
\begin{thm}
If the Goldbach-Intervals conjecture \ref{conj:intervals-Goldbach} is true then the Goldbach conjecture \ref{conj:(Binary-Goldbach)} is true.
\end{thm}

\begin{conjecture}
\label{conj:twin-Goldbach}\textbf{(Goldbach-Twin)}
For any positive even natural number $n$ there are at least two prime numbers $p\in\mathcal{P}$ and $q\in\mathcal{P}$ such that $m=p+q$ and $p$ or $q$ is a twin prime.    
\end{conjecture}

\begin{conjecture}
\label{conj:cousin-Goldbach}\textbf{(Goldbach-Cousin)}
For any positive even natural number $n$ there are at least two prime numbers $p\in\mathcal{P}$ and $q\in\mathcal{P}$ such that $m=p+q$ and $p$ or $q$ is a cousin prime.    
\end{conjecture}

\begin{conjecture}
\label{conj:intervals-Goldbach}\textbf{(Goldbach-Sexy)}
For any positive even natural number $n$ there are at least two prime numbers $p\in\mathcal{P}$ and $q\in\mathcal{P}$ such that $m=p+q$ and $p$ or $q$ is a sexy prime.    
\end{conjecture}

Goldbach-Twin, Goldbach-Cousin, and Goldbach-Sexy conjectures for the first $10$-primorial sets can be validated with the C++ program (\textit{onprimorials.cpp}) freely available at the prime numbers github repository of professor Jonatan Gómez \cite{GomezPrimesGit}.

\section{Conclusions and Future Work}
We have shown that primorial numbers may be useful for understanding some properties of prime numbers and prime numbers classes. For instance, we defined the concept of $n$-primorial set and established a relation between prime numbers in the interval $[p_n,\#(n)+1]$ and $n$-totative numbers. We used primorial intervals and primorial tables to define functions that count $n$-admissible $k$-tuples and to establish a relationship between them and their prime admissible $k$-tuples counterparts. In particular, we showed that the behavior of the ratio between the number of twin primes and the number of $n$-twins is the same as the behavior of the ratio between the number of cousin primes and the number of $n$-cousins, and the behavior of the ratio between the number of sexy primes and the number of $n$-sexy numbers. Finally, we stated variation on Goldbach's conjecture one in terms of primorial intervals and three in terms of twin, cousin, and sexy prime numbers. We computationally validate such conjectures for even numbers up to the 10th primorial number,i.e., up to $\#(10)=6469693230$, see C++ program \textit{onprimorial.cpp} at professor Jonatan Gomez github repository \cite{GomezPrimesGit}. Our future work will concentrate on expressing the Prime Number Theorem \cite{narkiewicz2000} in terms of both primorial numbers and $n$-totative numbers. Also, we will study the relationship between the asymptotic behavior of $n$-totative admissible $k$-tuples as $n$-twin, $n$-cousin, and $n$-sexy couples and their corresponding twin, cousin, and sexy prime counterparts. Finally, we will try to find proof of the stated conjectures.
\printbibliography

@book{narkiewicz2000,
    title = {The Development of Prime Number Theory: From Euclid to Hardy and Littlewood},
    author = {Narkiewicz, Władysław},
    isbn = {9783642085574},
    series = {Monographs on mathematics},
    year = {2000},
    publisher = {Springer}
}

@online{wikipediaprimelist,
    author = "Wikipedia",
    title = "List of Prime Numbers",
    url  = "https://en.wikipedia.org/wiki/List_of_prime_numbers",
    addendum = "(accessed: 23.11.2022)"
}

@online{wikipediaprimeconjectures,
    author = "Wikipedia",
    title = "List of Prime Numbers Conjectures",
    url  = "https://en.wikipedia.org/wiki/Category:Conjectures_about_prime_numbers",
    addendum = "(accessed: 23.11.2022)"
}

@online{wikipediaprimektuple,
    author = "Wikipedia",
    title = "Prime tuples",
    url  = "https://en.wikipedia.org/wiki/Prime_k-tuple",
    addendum = "(accessed: 23.11.2022)"
}

@online{GomezPrimesGit,
    author = {Gomez, Jonatan},
    title = {Prime Numbers Programs},
    url  = "https://github.com/jgomezpe/primenumbers",
    addendum = "(accessed: 5.1.2023)"
}

@article{https://doi.org/10.1112/plms/s2-44.4.307,
author = {Estermann, T.},
title = {On Goldbach's Problem : Proof that Almost all Even Positive Integers are Sums of Two Primes},
journal = {Proceedings of the London Mathematical Society},
volume = {s2-44},
number = {1},
pages = {307-314},
doi = {https://doi.org/10.1112/plms/s2-44.4.307},
url = {https://londmathsoc.onlinelibrary.wiley.com/doi/abs/10.1112/plms/s2-44.4.307},
eprint = {https://londmathsoc.onlinelibrary.wiley.com/doi/pdf/10.1112/plms/s2-44.4.307},
year = {1938}
}

@book{Rassias,
author = {Rassias, Michael},
year = {2017},
month = {01},
pages = {},
title = {Goldbach’s Problem},
isbn = {978-3-319-57912-2},
doi = {10.1007/978-3-319-57914-6}
}

@article{10.1090/S0025-5718-99-01117-5,
author = {Forbes, Tony},
title = {Prime Clusters and Cunningham Chains},
year = {1999},
issue_date = {Oct. 1999},
publisher = {American Mathematical Society},
address = {USA},
volume = {68},
number = {228},
issn = {0025-5718},
url = {https://www.ams.org/journals/mcom/1999-68-228/S0025-5718-99-01117-5/S0025-5718-99-01117-5.pdf},
doi = {10.1090/S0025-5718-99-01117-5},
journal = {Math. Comput.},
month = {oct},
pages = {1739–1747},
numpages = {9}
}

@article{Dubner1987,
author = {Dubner, Harvey},
title = {Factorial and primorial primes},
journal = {Recreational Math},
volume = {21(4)},
page={197–203},
year={1987}
}

@article{Euler1763,
author = {Euler, Leonhard},
title = {Theoremata arithmetica nova methodo demonstrata}, 
journal = {Novi commentarii academiae scientiarum imperialis Petropolitanae},
volume = {8},
pages={74–104},
year={1763}
}

\end{document}